\documentclass{amsart}
\usepackage{amsmath,amssymb,amsthm,amsfonts,amscd,mathrsfs,mathtools}
\usepackage{tikz}
\usepackage{braids}
\usepackage[hyphens]{url}
\usepackage{caption}
\usepackage{hyperref}

\usepackage{mymacros}

\newcommand{\thmref}[1]{Theorem~\ref{#1}}

\newcommand{\lemref}[1]{Lemma~\ref{#1}}
\newcommand{\corref}[1]{Corollary~\ref{#1}}

\theoremstyle{plain}
    \newtheorem{thm}{Theorem}[section]
    \newtheorem{lem}[thm]   {Lemma}
    \newtheorem{cor}[thm]   {Corollary}

\theoremstyle{definition}
    \newtheorem{defn}[thm]  {Definition}
    
    \newtheorem{conj}[thm]{Conjecture}

    \newtheorem{rem}[thm]{Remark}

\begin{document}

\title[TC of unordered configuration spaces of surfaces]{Topological complexity of unordered configuration spaces of surfaces}

\author{Andrea Bianchi}
\author{David Recio-Mitter}

\address{Mathematics Institute,
University of Bonn,
Endenicher Allee 60, Bonn,
Germany
}

\address{Institute of Mathematics,
Fraser Noble Building,
University of Aberdeen,
Aberdeen AB24 3UE,
UK}

\email{bianchi@math.uni-bonn.de}

\email{david.reciomitter@abdn.ac.uk}

\date{\today}


\begin{abstract}
We determine the topological complexity of unordered configuration spaces on almost all punctured surfaces (both orientable and non-orientable). We also give improved bounds for the topological complexity of unordered configuration spaces on all aspherical closed surfaces, reducing it to three possible values. The main methods used in the proofs were developed in 2015 by Grant, Lupton and Oprea to give bounds for the topological complexity of aspherical spaces. As such this paper is also part of the current effort to study the topological complexity of aspherical spaces and it presents many further examples where these methods strongly improve upon the lower bounds given by zero-divisor cup-length.

\end{abstract}


\maketitle

\section{Introduction}\label{sec:intro}

In 2003 Farber introduced the topological complexity of a space to study the problem of robot motion planning from a topological perspective \cite{Far03}.
It is a numerical homotopy invariant which measures the minimal instability of every motion planner on this space.
More explicitly, given a path-connected space $X$, the topological complexity $\TC(X)$ is the sectional category of the free path fibration $p_X:X^I\to X\times X$ (see section \ref{sec:aspherical}).

Determining $\TC(X)$ is in general a hard problem. For over a decade the topological complexity of many spaces has been computed and diverse tools have been developed to that end.

In this context, configuration spaces have been extensively studied because they are of special interest from the point of view of robotics. Considering the problem of moving $n$ objects on a space $X$ avoiding collisions naturally leads to the definition of the \emph{ordered configuration space} $F(X,n)$ of $n$ distinct ordered points on $X$ as

\begin{align*}
 F(X,n)=\{(x_1,\ldots,x_n)\in X^n\,|\,x_i\neq x_j\;\text{for}\;i\neq j\}.
\end{align*}

These spaces model Automated Guided Vehicles (AGVs) moving on a factory floor \cite{Ghr} or flying drones trying to avoid each other in the air.

Farber and Yuzvinsky determined the topological complexity of $F(\R^d,n)$ for $d=2$ or $d$ odd in \cite{FY}. Later Farber and Grant extended the results to all dimensions $d$ in \cite{FG}. The topological complexity of ordered configuration spaces of orientable surfaces has also been computed by Cohen and Farber in \cite{CF}. Many more related results can be found in the recent survey articles \cite{Coh} and \cite{Far17}.

In the configuration spaces $F(X,n)$ considered above, the points of a configuration are labelled (or ordered) and the symmetric group $\sym_n$ acts on $F(X,n)$ by permuting the labels. However, in certain situations it greatly improves the efficiency to consider the points to be identical. For instance, consider a scenario in which all the AGVs perform the same tasks equally. In this case we are only interested in the positions of points in $X$ up to permutation, in other words forgetting the labels assigned to the points. This leads to the \emph{unordered configuration spaces} $C(X,n)=F(X,n)/\sym_n$, by definition the orbits of the symmetric group action.

As we saw above, there is a very complete picture of the topological complexity of \textit{ordered} configuration spaces of 2-dimensional manifolds and beyond. In contrast to this, very little is known for \textit{unordered} configuration spaces, as Cohen notes at the end of his survey article \cite{Coh}. One of the main reasons for this discrepancy is that all the above results use a cohomological technique involving \emph{zero-divisors}, which seems to be insufficient for unordered configuration spaces (at least with constant coefficients).

The results in this paper use a technique to bound the topological complexity of aspherical spaces developed in 2015 by Grant, Lupton and Oprea \cite{GLO}. Being a homotopy invariant, the topological complexity of an aspherical space only depends on its fundamental group and the methods are algebraic in nature. An introduction to topological complexity of groups is given in section \ref{sec:aspherical}.

The mentioned technique was already used in the recent paper \cite{GRM}, in which Grant and the second author computed the topological complexity of some mixed configuration spaces $F(\R^2,n)/(\sym_{n-k}\times\sym_k)$ on the plane, with $1\le k\le n-1$. These spaces are in a sense intermediate between the ordered and the unordered case and they model the situation in which there are two different types of identical AGVs. It turns out that also in the mixed case the cohomological lower bounds used in previous results are insufficient.

It has to be mentioned that the topological complexity of unordered configuration spaces of trees was computed in many cases by Scheirer in \cite{Sch}. To the best of the authors' knowledge that is the only previous computation of the topological complexity of an unordered configuration space with at least three points. It is worth noting that Scheirer uses the zero-divisor cup-length lower bound which seems to be insufficient for unordered configuration spaces of surfaces.

In this paper we determine the topological complexity of the unordered configuration spaces of all punctured surfaces (orientable and non-orientable) except the disc and the M\"obius band, and narrow it down to three values for all closed aspherical surfaces (orientable and non-orientable). For the M\"obius band we narrow it down to two values and for the disc we give some improved bounds and a complete answer in the case of three points. Many of the proofs extend to ordered configuration spaces (this is discussed at the end of the paper).

All results except the ones for the disc are presented in the following theorem, which follows from the theorems \ref{thm:lowergeneral}, \ref{thm:lowerannulus}, \ref{thm:uppergeneral} and \ref{thm:upperannulus}. In the case of the annulus the upper bound is proven by finding an explicit motion planner.

\begin{thm}\label{thm:surfaces}
\begin{itemize}
\item Let $\S$ be obtained from a closed surface by removing a positive number of points. 
If $\S$ is not the disc, the annulus or the M\"obius band, then
\[\TC(C(\S,n))=2n.\]
\item Let $\S$ be a closed surface. If $\S$ is not the
sphere or the projective plane, then
\[2n\le\TC(C(\S,n))\le2n+2.\]

\item If $\A$ denotes the annulus, then
$$\TC(C(\A,n))=2n-1.$$

\item If $\M$ denotes the M\"obius band, then
$$2n-1\le\TC(C(\M,n))\le2n.$$

\end{itemize}
\end{thm}


\begin{rem}
\thmref{thm:surfaces} should be compared to the corresponding results for ordered configuration spaces in \cite{CF}. They are consistent with the possibility that the values of the topological complexity of ordered and unordered configuration spaces of surfaces always agree. Note that in \cite{CF} the non-reduced version of topological complexity is used, which is one greater than the one used in this paper.
\end{rem}

The only aspherical surface not covered by \thmref{thm:surfaces} is the disc. The best estimates we found for the disc are given in the following two theorems. Note that they greatly improve over the best previously known lower bounds
\[\TC(C(D,n))\ge \cat(C(D,n))=n-1\]
coming from the Lusternik-Schnirelmann category $\cat(C(D,n))$ (see \cite{GRM}).

\begin{thm}\label{thm:disc}
If $D$ is the disc, then
\[
2n-2-\frac{n}{2}\le n-1+\chd([P_n,P_n])\le\TC(C(D,n))\le2n-2.
\]
\end{thm}
Here $\chd$ is the cohomological dimension of a group and $[P_n,P_n]$ is the commutator subgroup of the pure braid
group of the disc (see section \ref{sec:braidgroups}).

We expect that $\chd([P_n,P_n])$ is in fact the maximum possible, which would mean that \thmref{thm:disc} narrows $\TC(C(D,n))$ down to two possible values.

\begin{conj}
The cohomological dimension of $[P_n,P_n]$ is equal to $n-2$.
\end{conj}

The following theorem gives a potentially better lower bound (depending on the actual value of $\chd([P_n,P_n])$, which is unknown to the authors). It also tells us that asymptotically $\TC(C(D,n))$ behaves like $2n$.

\begin{thm}\label{thm:discsqrt}
If $D$ is the disc, then
\[
2n-2\lfloor\sqrt{n/2}\rfloor-3\le\TC(C(D,n))\le 2n-2.
\]
\end{thm}

Finally, we compute the topological complexity of the unordered configuration space of three points on the disc by finding an explicit motion planner.

\begin{thm}\label{thm:threepoints}
If $D$ is the disc, then
\[
\TC(C(D,3))=3.
\]
\end{thm}

The authors are grateful to Mark Grant for many useful discussions and comments on earlier drafts of the paper, and to Gabriele Viaggi for suggesting the strategy for the proof of \lemref{lem:technical}.







\section{Topological complexity of aspherical spaces}\label{sec:aspherical}

In this section we first define the topological complexity of a general topological space and then specialize it to aspherical spaces.

For a path-connected topological space $X$, let $p_X:X^I\to X\times X$ denote the free path fibration on $X$, with projection $p_X(\gamma) = (\gamma(0),\gamma(1))$.

\begin{defn}
The \emph{topological complexity} of $X$, denoted $\TC(X)$, is defined to be the minimal $k$ such that $X\times X$ admits a cover by $k+1$ open sets $U_0,U_1,\ldots , U_k$, on each of which there exists a local section of $p_X$ (that is, a continuous map $s_i:U_i\to X^I$ such that $p_X\circ s_i = \mathrm{incl}_i:U_i\hookrightarrow X\times X$).
\end{defn}

Note that here we use the reduced version of $\TC(X)$, which is one less than the number of open sets in the cover.

Let $\pi$ be a discrete group. It is well-known that there exists a connected CW-complex $K(\pi,1)$ with
\[\pi_i(K(\pi,1))=\left\{\begin{array}{ll} \pi & (i=1) \\ 0 & (i\ge2). \end{array}\right.\]
Such a space is called an \emph{Eilenberg--Mac Lane space} for the group $\pi$. Furthermore, $K(\pi,1)$ is unique up to homotopy. Because the topological complexity $\TC(X)$ is a homotopy invariant of the space $X$ (see \cite{Far03}) the following definition is sensible.

 \begin{defn}
 The topological complexity of a discrete group $\pi$ is given by
 \[\TC(\pi)\coloneqq\TC(K(\pi,1)).\]
 \end{defn}

In \cite{Far06} Farber posed the problem of giving an algebraic description of $\TC(\pi)$. This problem is far from being solved but some progress has been made, including the following theorem.

 \begin{thm}[{Grant--Lupton--Oprea \cite[Theorem 1.1]{GLO}}]\label{thm:lowerbound}
 Let $\pi$ be a discrete group, and let $A$ and $B$ be subgroups of $\pi$. Suppose that $gAg^{-1} \cap B = \{1\}$
 for every $g \in \pi$. Then \[\TC(\pi)\ge \chd(A \times B).\]
 \end{thm}
 
 It is worth noting that this theorem has recently been generalised using different methods in \cite[Corollary 3.5.4]{FGLO}.

The corresponding problem for the Lusternik-Schnirelmann category of a group has been completely answered: $\cat(\pi)=\chd(\pi)$. This classical result is due to Eilenberg and Ganea \cite{EG} for $\chd(\pi)\neq1$, while the remaining case follows from the later work by Stallings \cite{Sta} and Swan \cite{Swa}.

We will also need the following standard result.

\begin{lem}\label{lem:upperbound}
$\TC(\pi)\le\chd(\pi\times\pi)$.
\end{lem}
\begin{proof}
This follows from the upper bound $\TC(X)\le\cat(X\times X)$ given by Farber in \cite{Far03}.
\end{proof}

\section{The surface braid groups}\label{sec:braidgroups}

In this section we introduce the surface braid groups and we recall their main properties.





\begin{defn}

A surface $\S$ is a connected closed 2-dimensional manifold possibly with a finite number of points removed, called \emph{punctures}.

\end{defn}



Recall from the introduction that the configuration space $F(\S,n)$ admits an action by the symmetric group $\sym_n$ which permutes the points in each configuration. The unordered configuration space 
\[
C(\S,n)=F(\S,n)/\sym_n
\]
is by definition the orbit space of that action.

\begin{defn}
We call $P_n(\S)=\pi_1(F(\S,n))$ the \emph{pure braid group} on $n$ strands
of the surface $\S$, and $B_n(\S)=\pi_1(C(\S,n))$ the \emph{(full) braid group} on $n$ strands of $\S$.
When $\S$ is the disc $D$, we also abbreviate $P_n=P_n(D)$ and $B_n=B_n(D)$.
\end{defn}

The covering $F(\S,n)\to C(\S,n)$ yields the short exact sequence
\begin{align*}
 1\to P_n(\S) \to B_n(\S)\to \sym_n\to 1.
\end{align*}

The following theorem is due to Fadell and Neuwirth.

\begin{thm}[Fadell-Neuwirth \cite{FN}]\label{thm:FN}
Denote by $\S_{n}$ the surface obtained from $\S$ by removing $n$ points. There is a locally trivial fibration
\begin{align}\label{eq:FNfibr}
\S_{n}\to F(\S,n+1)\to F(\S,n),
\end{align}
where the projection map forgets the last point of the ordered configuration. 
\end{thm}

It is well-known that the only surfaces that are not aspherical are the sphere $S^2$ and the projective plane $\R P^2$. From now on all the surfaces that we will consider are assumed to be aspherical.
The reason for this is that the methods in this paper only apply to aspherical spaces.

\begin{cor}\label{cor:FN}
Let $\S$ be an aspherical surface. From the long exact sequence of the homotopy groups applied to the Fadell-Neuwirth fibrations (\ref{eq:FNfibr}) and induction it follows that the spaces $F(\S,n)$ are also aspherical. Furthermore, we get the following short exact sequence.
\begin{align}\label{eq:FN}
 1\to \pi_1(\S_{n})\to P_{n+1}(\S)\to P_n(\S)\to 1
\end{align}
\end{cor}

We will need the following technical result, which we expect to be well-known to the experts. However, we could not find a full proof in the literature and thus we will give a detailed proof here. The result appears as Proposition 2.2 in \cite{PR} but it relies on \lemref{lem:technical} below (Proposition 2.1 in \cite{PR}), which is stated there without a proof. 




\begin{thm}
\label{thm:BSinjBT}
Let $\S\hookrightarrow\T$ be a smooth embedding of aspherical surfaces, such that the induced homomorphism $\pi_1(\S)\rightarrowtail\pi_1(\T)$ is injective.


Then the corresponding inclusion $C(\S,n)\hookrightarrow C(\T,n)$ induces an injective homomorphism $B_n(\S)\to B_n(\T)$.
\end{thm}

In the proof of the theorem the following lemma will be essential. In that lemma a slightly different definition of non-closed surface is needed, with open balls removed instead of points removed. This is the only place in which we make use of this definition. We stress that this is not an essential distinction because the configuration spaces of punctured surfaces and the configuration spaces of surfaces with boundary are homotopy equivalent.

\begin{lem}\label{lem:technical}
Let $\S\hookrightarrow\T$ be a smooth embedding of aspherical surfaces, which we assume to be closed surfaces with (possibly) some open balls removed instead of points removed. Further assume that the image of $\S$ lies in the interior of $\T$. Then the induced homomorphism $\pi_1(\S)\rightarrow\pi_1(\T)$ is injective if and only if no boundary
component of $\S$ bounds a disc in $\T\setminus\S$.
\end{lem}
\begin{proof}

Recall that we are assuming that surfaces are path-connected. Therefore, if $\S$ is closed the embedding has to be surjective and the claim is trivial. Assume $\S$ is not closed. Because the boundary components of $\S$ are smooth simple closed curves inside $\T$ they separate $\T$ into $\S$ on one side and a disjoint union of surfaces on the other side.

We first assume that the homomorphism $\pi_1(\S)\to\pi_1(\T)$ induced by the embedding is not injective and claim that
there is a disc in $\T\setminus\S$ bounded by a boundary component of $\S$.

A non-trivial element in the kernel of $\pi_1(\S)\to\pi_1(\T)$ can be represented by a smooth map $f\colon S^1\to\S$ which extends to a smooth map on the disc $g\colon D\to \T$. We may assume that the image of $f$ is in the interior of $\S$ and that $g$ is transverse to $\partial\S$.

Observe that the image of $g$ needs to have a non-empty intersection with the boundary of $\S$. Otherwise $g$ would yield a null-homotopy of $f$ inside $\S$, but by assumption $f$ represents a non-trivial class in $\pi_1(\S)$. Let $B$ be a boundary component of $\S$ which intersects the image of $g$.

The preimage of $B$ in $D$ under $g$ is now a non-empty, smooth 1-dimensional manifold. Since $f\colon S^1\to \S$ doesn't intersect
$\partial\S$, $g^{-1}(B)$ is a compact subset of the interior of $D$, hence it must be a
closed $1-$manifold.

Therefore given a path-component $C\subset D$ of $g^{-1}(B)$, we know that $C$ is a smooth circle and, by the Jordan-Schoenflies
curve theorem, $C$ bounds a disc $\tilde D$ in $D$ on one side and an annulus $A$ on
 the other side, such that $\partial A=C\cup \partial D$.
We can further assume, by choosing $C$ to be \emph{outermost} in $D$ among the path-components of $g^{-1}(B)$,
that there exists a collar neighborhood $U\supset C$ in $D$ such that $g(U\cap A)\subseteq \S$. Indeed, by transversality
we have, for a small collar neighborhood $U$, that $g(U\cap A)$ is either contained in $\S$ or in $\T\setminus \S$.
If $C$ is \emph{outermost} the former must be the case, as under this condition there is a path in $A$ from
$C$ to $\partial D$ only intersecting $g^{-1}(B)$ at the starting point, and $g(\partial D)\subset \S$.

The curve $C$ gives an element in $\pi_1(B)\simeq \Z$. If this element is trivial then we can redefine $g$ on $\tilde D$ by a nullhomotopy living on $B$. After pushing the image of $\tilde D$ along the collar neighborhood
into the interior of $\S$, we get a replacement of $g$ with (at least) one fewer connected component in $g^{-1}(\partial\S)$ than for the original map.

Hence there must exist a circle $C$ such that $g|_C$ is a non-trivial element in $\pi_1(B)$, otherwise
we would construct a nullhomotopy of $f$ inside $\S$ after finitely many iterations of the above procedure. Therefore, there is a power of the generator $[B]\in\pi_1(B)$ that vanishes in $\pi_1(\T)$. Because $\pi_1(\T)$ is torsion-free
(indeed $\T$ is a finite-dimensional classifying space for $\pi_1(\T)$), $[B]$ is already trivial in $\pi_1(\T)$.

Then $B$ is a null-homotopic simple closed curve and it must bound a disc in $\T$ by the classification of surfaces. There are two possibilities. Either this disc doesn't intersect the interior of $\S$ and it is glued to the boundary component $B$ to obtain $\T$, or $\S$ is a punctured sphere and $\T$ is obtained from $\S$ by glueing discs onto all the path-components of $\partial\S$ different from $B$ (there is at least one other boundary component because by assumption $\pi_1(\S)\to\pi_1(\T)$ is not injective and therefore $\S$ is not a disc).

We showed that if $\pi_1(\S)\to\pi_1(\T)$ is not injective, there must be a disc in $\T\setminus \S$ bounded by
boundary component of $\S$.

Conversely, assume that $\T\setminus \S$ contains a disc $D$ bounded by
some boundary component $B$ of $\partial \S$.
Then the corresponding element $[B]\in\pi_1(\S)$ vanishes in $\pi_1(\T)$. Therefore, the homomorphism $\pi_1(\S)\to\pi_1(\T)$ is not injective
unless $[B]$ is already trivial in $\pi_1(\S)$. Again by the classification of surfaces, this can only happen if $\S$ itself is a disc,
but then $\T$ would be a sphere, contradicting the hypothesis that $\T$ is aspherical.
\end{proof}

\begin{proof}[Proof of \thmref{thm:BSinjBT}]

By the commutativity of the following diagram with exact rows, it suffices to show that $P_n(\S)\to P_n(\T)$ is injective.
\begin{align*}
  \begin{CD}
 1 @>>>  P_n(\S)@>>> B_n(\S) @>>>\sym_n @>>> 1\\
 @.  @VVV @VVV @| @.\\
 1 @>>>  P_n(\T)@>>> B_n(\T) @>>>\sym_n @>>> 1\\
  \end{CD}
 \end{align*}

 We do this by induction using the Fadell-Neuwirth fibrations.
 
 For $n=1$, the homomorphism $\pi_1(\S)\to\pi_1(\T)$ is injective by assumption.
 
 Suppose now that $P_{n-1}(\S)\to P_{n-1}(\T)$ is injective. The embedding $\S\hookrightarrow \T$ gives rise to an embedding $\S_n\hookrightarrow \T_n$, in which the $n$ new punctures in $\S_n$ are sent to the $n$ new punctures in $\T_n$. The short exact sequences (\ref{eq:FN}) give rise to the following commutative diagram.
 \begin{align*}
  \begin{CD}
   1@>>>\pi_1(\S_{n-1})@>>> P_n(\S) @>>>P_{n-1}(\S)@>>>1\\
   @. @VVV @VVV @VVV\\
   1@>>>\pi_1(\T_{n-1})@>>> P_n(\T) @>>>P_{n-1}(\T)@>>>1 
  \end{CD}
 \end{align*}
The rows are exact and we assumed the vertical homomorphism on the right is injective. If the vertical homomorphism on the left were also injective, the vertical homomorphism in the middle would have to be injective, which would complete the induction argument.

It is not hard to see that the configuration spaces of punctured surfaces (points removed) and the configuration spaces of surfaces with boundary components (open balls removed) are homotopy equivalent. Because of this we might assume that $\S$ and $\T$ are surfaces with boundary and that $\S_{n-1}$ and $\T_{n-1}$ are the surfaces which result from removing $n-1$ open balls, in order to be able to use \lemref{lem:technical}. Then the embedding $\S\hookrightarrow\T$ satisfies the assumptions of \lemref{lem:technical} if and only if $\S_{n-1}\hookrightarrow\T_{n-1}$ satisfies them. Therefore, the injectivity of the leftmost vertical homomorphism is equivalent to the injectivity of $\pi_1(\S)\rightarrowtail\pi_1(\T)$, which is part of the assumptions.
\end{proof}

\section{Lower bounds}\label{sec:lower}

\begin{thm}\label{thm:lowergeneral}
Let $\mathcal{S}$ be an aspherical surface which is not the disc, the annulus or the M\"obius band. Then

$$
\TC(C(\mathcal{S},n))\ge2n.
$$
\end{thm}

\begin{proof}
Let $\mathcal{S}$ be a surface satisfying the assumptions in the theorem. Then, with the only exception of the Klein bottle, we have $\rk(H_1(\mathcal{S}))\geq 2$ and there are two smooth simple closed curves $\alpha$ and $\alpha'$ on $\mathcal{S}$ representing linearly independent classes of $H_1(\mathcal{S})$. We may assume that there exist tubular neighborhoods
$\A$ of $\alpha$ and $\A'$ of $\alpha'$ that are annuli. If the tubular neighborhood of $\alpha$ were a M\"obius band, then we could replace $\alpha$ by the boundary of this M\"obius band.

The homomorphism $\pi_1(\A)\to \pi_1(\mathcal{S})$ is injective,
as can be checked by further projecting to $H_1(\mathcal{S})$. Similarly the homomorphism $\pi_1(\A')\to\pi_1(\mathcal{S})$ is injective.

For the Klein bottle $\K$, recall that the fundamental group $\pi_1(\K)$ has a presentation
\[
 \pi_1(\K)=\left< a,b\,|\,aba^{-1}b\right>,
\]
where both $a$ and $b$ are represented by simple closed curves $\alpha$ and $\beta$ in $\K$. Both subgroups $\left<a\right>$
and $\left<b\right>$ are
infinite cyclic, and therefore the inclusions of collar neighborhoods $\A$ of $\alpha$ and $\A'$ of $\beta$ in $\K$ are
injective at the level of $\pi_1$; the collar neighborhood of $\alpha$ is a M\"obius band so we replace $\alpha$ with its
\emph{double} as above.

Hence by \thmref{thm:BSinjBT} the homomorphisms $P_n(\A)\to P_n(\mathcal{S})$ and $P_n(\A')\to P_n(\mathcal{S})$ are injective.


We now construct a subgroup $Z_n\subset P_n(\A)$. Consider $n$ parallel, disjoint copies $\alpha_1,
\dots,\alpha_n$ of 
the curve $\alpha$ inside $\A$, and let $\mathfrak{T}\subset F(\A,n)$ be the subspace of ordered
configurations $(x_1,\dots,x_n)$ with $x_i$ lying on the curve $\alpha_i$ for all $1\leq i\leq n$;
then $\mathfrak{T}$ is an embedded $n$-fold torus in $F(\A,n)$, and at the level of fundamental groups we have
a map $\Z^n\simeq\pi_1(\mathfrak{T})\to P_n(\A)$.

This map is injective: indeed the composition
\[
 \Z^n\simeq\pi_1(\mathfrak{T})\to P_n(\A)=\pi_1(F(\A,n))\to \pi_1(\A^n)\simeq\Z^n
\]
is an isomorphism. We call $Z_n\simeq\Z^n\subset P_n(\A)$ the image of this map.

In the same way we construct an $n-$fold torus $\mathfrak{T}'\subset F(\A',n)$ and get a subgroup
$Z'_n\subset P_n(\A')$ as the image of the map between fundamental groups induced by the inclusion,
with $Z'_n\simeq\Z^n$.

\begin{figure}[htb!]
  \centering
  \includegraphics[scale=0.7]{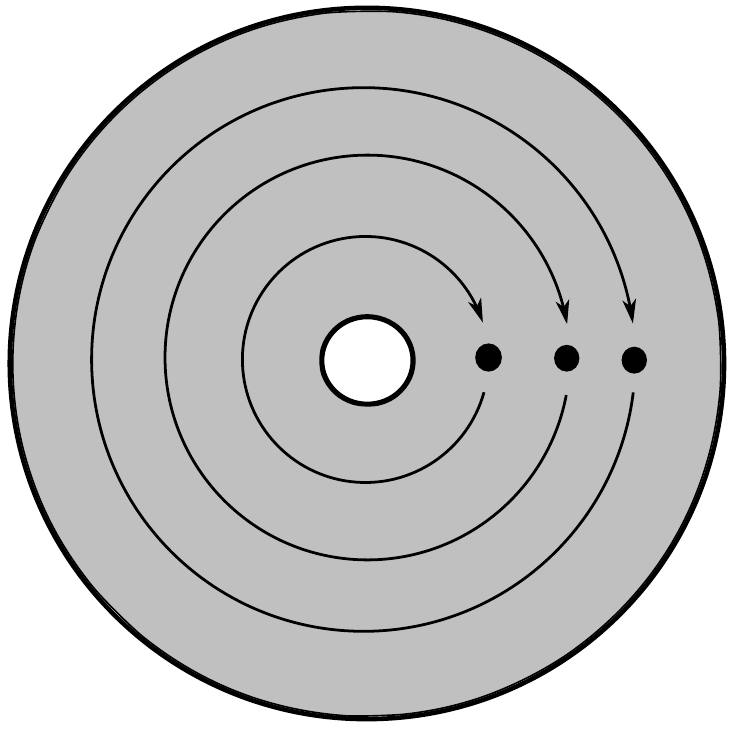}
  \caption{Braids from $Z_n$ as seen from above.}
  \label{fig:concentric}
\end{figure}


There is a homomorphism

\begin{align}\label{eq:fakeab}
P_n(\mathcal{S})\to\prod_{k=1}^n\pi_1(\mathcal{S})\to\bigoplus_{k=1}^nH_1(\mathcal{S}),
\end{align}

under which non-trivial elements in the image of $Z_n$ and $Z'_n$ inside $P_n(\mathcal{S})$
are mapped to elements which lie in different orbits under the action which permutes
the summands in $\bigoplus_{k=1}^nH_1(\mathcal{S})$. This is because the image of each non-trivial
element in $Z_n$ will have at least one summand corresponding to a non-trivial multiple
of the class in $H_1(\mathcal{S})$ represented by the curve $\alpha$, whereas the image of each braid
in $Z_n'$ has only summands corresponding to multiples of the class represented by the curve $\alpha'$.
Notice that for the Klein bottle it suffices that the homology class represented by $\alpha$ is infinite cyclic,
and the argument works even if the homology class represented by $\alpha'$ has order 2.

Now we observe that conjugating an element of $P_n(\mathcal{S})$ by an element of $B_n(\mathcal{S})$ has
the effect of permuting the summands in $\bigoplus_{k=1}^nH_1(\mathcal{S})$ under the homomorphism (\ref{eq:fakeab}). To see this first note that the homomorphism (\ref{eq:fakeab}) consists of a sum of compositions of homomorphisms of the form \[P_n(\mathcal{S})\to \pi_1(\mathcal{S})\to H_1(\mathcal{S})\] given by forgetting all strands but one and then taking the abelianization. Given a braid $\gamma\in B_n(\mathcal{S})$, we can write $\gamma=\delta\epsilon$, where $\epsilon$ is supported on a disc and $\delta\in P_n(\S)$. Therefore, conjugating by $\gamma$ reduces to conjugating by $\epsilon$ and $\delta$. Conjugating by $\epsilon$ permutes the order of the strands by the corresponding permutation under the canonical map $B_n \to \sym_n$. Conjugating by $\delta$ results in a conjugation inside $\pi_1(\mathcal{S})$ under the first homomorphism $P_n(\mathcal{S})\to \pi_1(\mathcal{S})$, but this has no effect on the abelianization.

Therefore, no non-trivial element of $Z_n$ is conjugate to an element of $Z_n'$ in $B_n(\mathcal{S})$. By \thmref{thm:lowerbound} this implies the lower bound $\TC(B_n(\mathcal{S}))\ge\chd(Z_n\times Z'_n)=2n$.
\end{proof}

\begin{thm}\label{thm:lowerannulus}
Let $\mathcal{S}$ be either the annulus or the M\"obius band. Then

$$
\TC(C(\mathcal{S},n))\ge2n-1.
$$
\end{thm}

\begin{proof}

In the same way as in the previous proof we can find an annulus $\A$ inside $\S$ and a subgroup $Z_n$ in $P_n(\A)$ isomorphic to $\Z^n$. Because $\pi_1(\S)\simeq H_1(\S)\simeq\Z$, this time we cannot find a second annulus inducing a linearly independent homology class, not even a disjoint infinite cyclic subgroup of $\pi_1(\S)$.

However, the inclusion of a disc $D$ in $\mathcal{S}$ also induces a monomorphism $P_n(D)\to P_n(\S)$ and no non-trivial element in $P_n(D)$ is conjugate to an element of $Z_n$ inside $B_n(S)$.


Indeed, if we consider the map
\[
P_n(\mathcal{S})\to\prod_{k=1}^n\pi_1(\mathcal{S})\to\bigoplus_{k=1}^nH_1(\mathcal{S}),
\]
we see that no non-trivial element of $Z_n$ is mapped to zero, whereas all elements of $P_n(D)$ are
mapped to zero. As we saw in the proof of the previous theorem, conjugation inside $B_n(S)$ results only in a permutation of the coordinates
of the target group $\bigoplus_{k=1}^nH_1(\mathcal{S})$, and the stated properties are therefore
invariant under conjugation.

By \thmref{thm:lowerbound} we get
\[
 \TC(C(\mathcal{S},n))\ge \chd(\Z^n\times P_n(D))= 2n-1.\qedhere
\]
\end{proof}

%
%
%
%
%
%
%

\section{Upper bounds}\label{sec:upper}

\begin{thm}\label{thm:uppergeneral}

If $\mathcal{S}$ is a closed aspherical surface, then

$$
\TC(C(\mathcal{S},n))\le2n+2.
$$

If $\mathcal{S}$ is a punctured surface which is not the disc, then

$$
\TC(C(\mathcal{S},n))\le2n.
$$

\end{thm}

\begin{proof}
  It is  well-known that $\chd(\pi_1(\mathcal{S}))=2$ for closed aspherical surfaces and $\chd(\pi_1(\mathcal{S}))=1$
  for punctured surfaces (other than the disc).
  Using the short exact sequences (\ref{eq:FN}) of \corref{cor:FN}, together with the fact that the
  cohomological dimension is subadditive under group
  extensions, and that $\chd(B_n(\S))=\chd(P_n(\S))$ because $B_n(\S)$ is torsion-free and $P_n(\S)$ is a finite index subgroup, we see that $\chd(B_n(\S))\leq n+1$ if $\S$ is closed and $\chd(B_n(\S))\leq n$ if $\S$
  has punctures and is not the disc (the two preceding inequalities are in fact equalities, but we don't
  need that stronger statement in this proof).
  
  The upper bounds now follow from \lemref{lem:upperbound}.
\end{proof}

Next we give an upper bound for the annulus which is one better than the one given in the previous theorem (it is in fact the optimal upper bound). For the proof we will need the following well-known technical lemma.

We defined the topological complexity in terms of the number of open sets in an open cover of $X\times X$ but for sufficiently nice spaces (CW-complexes for instance) there is an equivalent characterization in terms of decompositions into disjoint \textit{Euclidean neighborhood retracts} (ENRs).

\begin{lem}[\cite{Far06}]\label{lem:enr}

Let $X$ be an ENR (for instance a finite-dimensional locally finite CW-complex). Then the topological complexity $\TC(X)$ equals the smallest integer $k$ such that there exists a decomposition $X\times X=E_0\sqcup E_1\sqcup\cdots\sqcup E_k$ into $k+1$ disjoint ENRs, on each of which there is a local section $s_i\colon E_i\to X^I$.

The existence of such a section $s_i\colon E_i\to X^I$ is equivalent to the existence of a deformation of $E_i$ into the diagonal of $X\times X$, i.e. a homotopy between the inclusion $E_i\hookrightarrow X\times X$ and a map whose image lies entirely in the diagonal.

\end{lem}

\begin{thm}\label{thm:upperannulus}

If $\A$ is the annulus, then

$$
\TC(C(\A,n))\le2n-1.
$$

\end{thm}

\begin{proof}

By \lemref{lem:enr} we need to find a decomposition of $C_n(\A)\times C_n(\A)$ into $2n$ disjoint ENRs which can be deformed into the diagonal. Note that such deformations can equivalently be viewed as an explicit motion planner with $2n$ different continuous rules and as such it is potentially relevant for applications.

\subsection{Decomposition of $C_n(\A)\times C_n(\A)$}

The annulus can be identified with a product $\A=S^1\times \R$ of a circle and the real line. The projection map $p\colon\A\to S^1$ induces a map 
\[
 p_n\colon C_n(\A)\to Sym_n(S^1)
\]
where the latter space is the $n$-fold symmetric power of $S^1$, defined as the quotient of $(S^1)^{\times n}$ by
the action of $\mathfrak{S}_n$ on the coordinates.


For a given pair of configurations $(x,y)\in C_n(\A)\times C_n(\A)$ we interpret $p_n(x)$ and $p_n(y)$ as finite
subsets of $S^1$, i.e. we forget the multiplicities of points in $S^1$. The cardinality $\text{deg}(x,y)=|p_n(x)\cup p_n(y)|$ of the union of those subsets will be called the \emph{degree} of the pair.

Notice that $\text{deg}(x,y)$ is at least $1$ and at most $2n$. This yields a decomposition of $C_n(\A)\times C_n(\A)$ into $2n$ disjoint subspaces $L_k=\text{deg}^{-1}(k)$ corresponding
to the different values of $\text{deg}$, see Figure \ref{fig:annulus}. Furthermore, $L_k$ is a smooth embedded manifold and in particular an ENR.

\begin{figure}[htb!]
  \centering
  \includegraphics[scale=0.9]{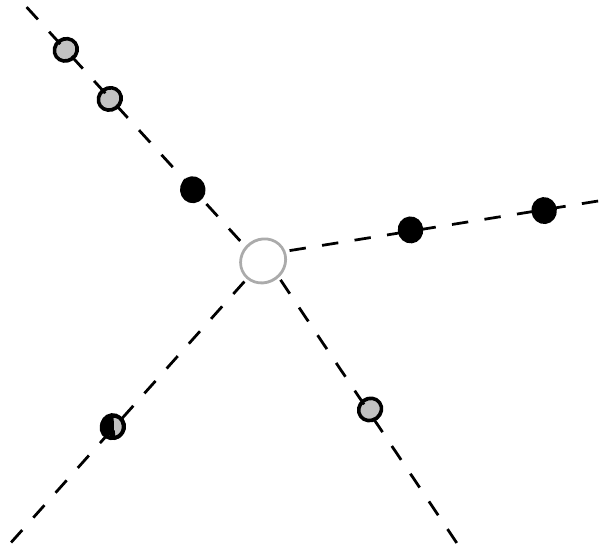}
  \captionof{figure}{A pair of configurations in $L_4$, with one
    double point.}
  \label{fig:annulus}
\end{figure}



\subsection{Local motion planners}

Given a pair $(x,y)\in L_k$, the union $p_n(x)\cup p_n(y)$ contains exactly $k$ distinct points $q_1,\dots, q_k\in S^1$, ordered cyclically on $S^1$ in clockwise direction. We need to introduce some notation. Let $n_{x,i}$ be the number of points in $x$ mapped to $q_i$ under $p$ and let $n_{y,i}$ be the number of points in $y$ mapped to $q_i$ under $p$. Finally, let $\delta_i=n_{x,i}-n_{y,i}$ be the difference between those two numbers.

The following map is continuous and well-defined.
\begin{align*}
\psi_k\colon &L_k\to\left\{(m_i)_i\in\Z^k\,|\, \sum^k_{i=1} m_i=0, \sum^k_{i=1}|m_i|\leq 2n\right\}/(12\ldots k)\\ &(x,y)\mapsto[(\delta_i)_i]
\end{align*}
Here $(12\ldots k)\in\mathfrak{S}_k$ is the \emph{long cycle}, permuting the components $m_i$.

Because the preimages of different $[(\delta_i)_i]$ are topologically disjoint, we can define the local section of the free path fibration over $L_k$ separately on each preimage.

Given a pair of configurations $(x,y)\in L_k$ lying in the preimage $\psi_k^{-1}([(\delta_i)_i])$ we need to construct a path between them, continuously over $L_k$.

If $\delta_i=0$ for all $i$ we will simply move the points of $x$ onto the points of $y$ on each fiber of $p$ by linear interpolation
inside the fibers.

On the other hand, if there exists an $i$ such that $\delta_i\neq0$, first we need to construct a path from $x$ to $\tilde x$ such that $\text{deg}(\tilde x, y)=\tilde k$ for some $\tilde k\leq k$, and such that $(\tilde x,y)\in\psi_{\tilde k}^{-1}((0)_i)$; then we concatenate this path with the fiberwise linear interpolation used above. The path from $x$ to $\tilde x$ will consist in an iteration of one particular deformation which we describe in the following and which is illustrated in Figure \ref{fig:annulus-motion}.

\subsection{First step}

Let  $(x,y)\in L_k$  as above and let $x$ consist of the points $x_{i,l}\in\A$ for $1\leq i\leq k$ and $1\leq l\leq n_{x,i}$, where for each $i$ the points $x_{i,l}$ are exactly those lying over $q_i\in S^1$ and the indices are chosen according to the order of the points on the fiber $p^{-1}(q_i)\simeq \R$.

We are going to deform $x$ into another configuration denoted $x^{(1)}$.

Whenever $\delta_i>0$, we move the $\delta_i$ top points of $x$ in $p^{-1}(q_{i})$ clockwise until they reach $p^{-1}(q_{i+1})$, on top of all points of $x$ already in $p^{-1}(q_{i+1})$ (if any). More precisely, we move the points $x_{i,l}$ for $n_{y,i}+1\leq l\leq n_{x,i}$ to $p^{-1}(q_{i+1})$
so as to keep their order and their pairwise distances, and such that $x_{i,n_{y,i}+1}$ reaches the position $1+\max\set{0,x_{i+1,n_{x,i+1}}}$ inside the fiber $p^{-1}(q_{i+1})\simeq \R$. We move these points by linear interpolation along the interval $[q_i,q_{i+1}]\subset S^1$ and along $\R$. We do this simultaneously for all $i$ for which $\delta_i>0$. Note that the indices are considered modulo $k$. This is shown in Figure \ref{fig:annulus-motion}.

It is clear from the construction that this deformation is continuous within $\psi_k^{-1}([(\delta_i)_i])$.

\begin{figure}[htb!]
  \centering
  \includegraphics[scale=0.9]{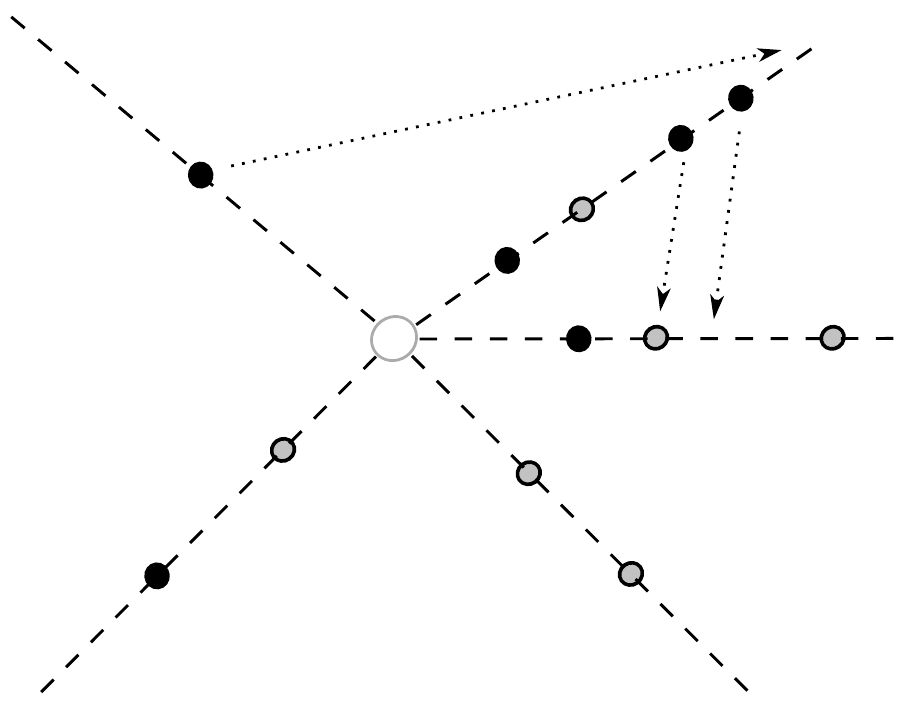}
  \captionof{figure}{One iteration of the motion planner on the
    annulus. Notice that the positions of the grey points on a given fibre are disregarded
    when moving black points towards it because the points exist in two separate spaces.}
  \label{fig:annulus-motion}
\end{figure}

\subsection{Iterations of the first step} We started with a pair of configurations $(x,y)\in L_k$ and in the previous subsection we constructed a deformation of $x$ into $x^{(1)}$. Clearly $k_1=\text{deg}(x^{(1)},y) \leq k$. We can now repeat the process starting with the pair $(x^{(1)},y)$ to get a new configuration $x^{(2)}$, again without changing $y$. Iterating this,
we get a sequence of configurations $x^{(j)}$ and a sequence of degrees $k_j=\text{deg}(x^{(j)},y)$, which is weakly
decreasing.

If this algorithm terminates after $T$ steps, then it gives us a path from $(x,y)\in L_k$ to $(x^{(T)},y)\in\psi_{\tilde k}^{-1}((0)_i)$. Furthermore, because each iteration is continuous it yields a continuous deformation of $L_k$ into $\psi_{\tilde k}^{-1}((0)_i)$, which completes the proof.

To see that the algorithm does indeed terminate, note that there exists an $N\in\N$ such that $k_j=k_N$ for all $j\geq N$. After $k_N$ further iterations we have that $\delta_{i}^{(N+k_N)}=0$ for all $1\leq i\leq k_N$ and we are done. This follows from the following three facts, which are easy to check.
\begin{enumerate}
\item For all $j> N$, if $\delta_i^{(j)}>0$, then $\delta_{i-1}^{(j-1)}>0$.
\item For all $j\geq N$, if $\delta_i^{(j)}\ge0$, then $\delta_i^{(j+1)}\ge0$.
\item For all $j\ge N$ we have $\sum_i\delta^{(j)}_i=0$.
\end{enumerate}
Indeed, if $\delta_i^{(N+k_N)}\neq0$ for some $i$ we may assume that $\delta_i^{(N+k_N)}>0$ because of (3). By (1) this would imply that $\delta_{i-l}^{(N+k_N-l)}>0$ for all $0\leq l\leq k_N-1$ and therefore by (2) $\delta_{i-l}^{(N+k_N)}\geq 0$ for all $0\leq l\leq k_N-1$.

This would mean that $\sum_{l=1}^{k_N-1}\delta^{(N+k_N)}_{i-l}>0$, which contradicts (3).
\end{proof}

\section{Proof of Theorems \ref{thm:disc} and \ref{thm:discsqrt}}\label{thmdisc}


\begin{proof}[Proof of \thmref{thm:disc}]
Notice that the pure braid group on the disc $P_n$ is isomorphic to the pure braid group $P_{n-1}(\A)$ on the annulus with one less strand in the following way. Every braid in $P_n$ can be chosen such that the last strand does not move and that strand is identified with the central hole in the annulus.

Consider the subgroup $Z_{n-1}\cong \Z^{n-1}\le P_{n-1}(\A)\cong P_n$ given by the braids in $P_{n-1}(\A)$ in which all strands move in concentric circles around the central hole.

Recall that for the abelianization $P_n^{ab}\cong \Z^{n \choose 2}$. The abelianization homomorphism is given by the collection over all unordered pairs $\{i,j\}$ of the maps $\psi_{i,j}\colon P_n\to P_2\cong \Z$ forgetting all strands except the $i$\textsuperscript{th} and the $j$\textsuperscript{th} (measuring the linking number between the strands $i$ and $j$).

Conjugating by an element $g\in B_n$ is compatible, under the abelianization, with the induced permutation
of the components $\psi_{i,j}$ of $P_n^{ab}$ coming from the canonical permutation in $\mathfrak{S}_n$ associated to $g$.

In light of the above, it is clear that the commutator subgroup $[P_n,P_n]$ (the kernel of the abelianization homomorphism) is not only normal in $P_n$ but also in $B_n$. Furthermore, it also follows that $Z_{n-1}$ is mapped injectively under the abelianization homomorphism and thus has a trivial intersection with $[P_n,P_n]$. Taken together this implies that the conjugates of a non-trivial element of $[P_n,P_n]$ cannot lie in $Z_{n-1}$.

The lower bound now follows from \thmref{thm:lowerbound} together with \lemref{lem:comm} below.

Finally, the upper bound follows from \lemref{lem:upperbound} and $\chd(B_n)=n-1$, which can be shown using the Fadell-Neuwirth fibrations as for the other aspherical surfaces.
\end{proof}

\begin{proof}[Proof of \thmref{thm:discsqrt}]

Let $(p_1,\dots, p_n)\in F(D,n)$ denote an ordered configuration of $n$ points in the disc $D$ and let $1\leq k\leq n$, to be chosen suitably later. Recall that based loops in $F(D,n)$ represent braids in the pure braid group $P_n$ and let $A\subset P_n$ consist of those pure braids represented by loops in which the points $p_1,p_2,\dots,p_k$
are fixed \emph{in the middle} and $p_{k+1},\dots,p_n$ independently rotate around this
cluster in concentric orbits. Clearly we have $A\cong\Z^{n-k}$.

We now write $n=mk+r$ for appropriate $m\geq 0$ and $1\leq r\leq k$. Notice that $r$ is assumed to be positive. 

Divide the points $p_1,\dots, p_n$ into $m$ clusters of $k$ points each plus an additional cluster of $r$ points. Let $B$ be the subgroup of $P_n$ in which points of the same cluster interact freely and such that moreover the $m+1$ clusters are allowed to move around each other,
so long as they don't mix and their trajectories describe an element in $[P_{m+1},P_{m+1}]$.

More formally, let $E_2(m+1)$ be the space
of ordered configurations of $m+1$ little discs $D_1,\dots, D_{m+1}$ inside the disc $D$. Each disc $D_i$ is uniquely determined by its
centre and its (positive) radius and the little discs are required to have disjoint interiors (see \cite{May} for an introduction
to the operad of little cubes). There is a map
\[
 E_2(m+1)\times F(D,k)\times\dots\times F(D,k)\times F(D,r)\to F(D,n)
\]
given by embedding each configuration of $k$ or $r$ points into the corresponding disc $D_i$, using
the only positive rescaling of $D$ onto $D_i$. Because $E_2(m+1)$ is also a classifying space for $P_{m+1}$, there is a homomorphism on fundamental groups
\[
\gamma\colon P_{m+1}\times P_k\times\dots \times P_k\times P_r\to P_n.
\]
To show that $\gamma$ is injective, let $\rho$ be the product of the following $m+2$ maps:
\begin{itemize}
 \item One map $P_n\to P_{m+1}$ given by forgetting all strands but a chosen one in each cluster, such that exactly
 $m+1$ strands remain.
 \item The maps $P_n\to P_k$ and $P_n\to P_r$ given by forgetting all strands outside a given cluster.
\end{itemize}
It is easy to see that $\rho$ is a retraction of $\gamma$ and that therefore
$\gamma$ is injective. The subgroup $B\subset P_n$ is defined to be the image of the restriction of $\gamma$ to $[P_{m+1},P_{m+1}]\times (P_k)^m\times P_r$.

Next we need to check that $A$ and $B$ satisfy the assumptions of \thmref{thm:lowerbound} as subgroups of $B_n$, i.e. $gAg^{-1}\cap B=\{1\}$ for all $g\in B_n$. For this we will use the abelianization of the pure braid group $P_n$. As we saw in the proof of \thmref{thm:disc}, the abelianization detects the pairwise linking numbers between the braids and conjugation by $g\in B_n$ permutes those numbers by the induced permutation.



The following property of an element $\sigma\in P_n$ is invariant under conjugation by each $g\in B_n$.

\emph{There exists an index $1\leq j\leq n$ and $k$ other indices $i_1,\dots, i_k$ such that
$\psi_{j,i_l}(\sigma)\neq 0$ for all $1\leq l\leq k$}.



Let $\alpha\in A$ be a non-trivial braid. In such a braid there is at least one point $p_j$, for $k+1\leq j\leq n$,
which rotates a non-zero number of times around the points $p_1,\ldots,p_k$. Therefore, the numbers $\psi_{l,j}(\alpha)$ are all non-zero (and equal to each other) for $1\leq l\leq k$.

However, no braid $\beta\in B$ has the property above. Indeed, $\psi_{i,j}(\beta)$ can be non-zero only if $p_i$ and $p_j$ are
in the same cluster, and every cluster contains at most $k$ points.

Hence we get that for each $1\leq k\leq n$
\[
\begin{split}
 TC(B_n) & \geq\chd(A\times B)\\
  & =n-k+m(k-1)+r-1+cd([P_{m+1},P_{m+1}])\\
  &\overset{\lemref{lem:comm}}{\geq} 2n-k-m-1+\frac {m-1}2\\
  &= 2n- k -\frac m2 -\frac 32.
 \end{split}
\]

Choosing $k=\lfloor\sqrt{n/2}\rfloor$, the inequality $$n=mk+r\geq mk+1$$ implies that $$m\leq (n-1)/k$$ and so $$m\leq \lfloor(n-1)/k\rfloor\leq 2k+4$$ by the choice of $k$. Therefore

\[
TC(B_n)\geq 2n-2\lfloor\sqrt{n/2}\rfloor -3-\frac 12
\]
and since $TC(B_n)$ is an integer we can drop the term $\frac 12$.
\end{proof}

\begin{lem}\label{lem:comm}
Let $[P_n,P_n]$ be the commutator subgroup of the pure braid group $P_n$. Then
\[
\chd([P_n,P_n])\ge\frac{n-2}{2}.
\]
\end{lem}

\begin{proof}

Like in the previous proof, let $E_2(3)$ denote the space of ordered configurations of 3 little discs $D_1$, $D_2$ and $D_3$ inside a disc $D$. There exists a map

\[
 E_2(3)\times F(D,3)\to F(D,5),
\]
given by embedding the configurations in $F(D,3)$ into the first disc $D_1$ (after the appropriate rescaling) and by mapping the other two little discs to their centre points.

Iterating this construction $k-1$ times results in the following map.

\[
\underbrace{E_2(3)\times E_2(3)\times\dots\times E_2(3)}_{k-1}\times F(D,3)\to F(D,2k+1)
\]

On fundamental groups this yields a homomorphism

\begin{align}\label{matryoshka}
P_3^k\to P_{2k+1}.
\end{align}

Similarly to the previous proof, this homomorphism is injective. By construction the images of the different $P_3$ factors commute with each other.

Let $\Z$ be an infinite cyclic subgroup of $[P_3,P_3]$. The image of the homomorphism (\ref{matryoshka}) restricted to $\Z^k\le P_3^k$ is isomorphic to $\Z^k$ by the above observations and it is a subgroup of $[P_{2k+1},P_{2k+1}]$. By the well-known properties of cohomological dimension

\[
\chd([P_{2k+1},P_{2k+1}])\ge\chd(\Z^k)=k.
\]

This proves the claim for $n=2k+1$ odd. For $n$ even the claim immediately follows from $P_{n-1}\le P_n$.
\end{proof}

\section{Motion planner for the disc}\label{sec:discmotion}

Let $D$ be the disc. In this section we are going to give an explicit motion planner which will imply that $\TC(C(D,3))=3$ as stated in \thmref{thm:threepoints}. Observe that a motion planner on a subset of $X\times X$ is the same as a deformation into the diagonal.

\begin{proof}[Proof of \thmref{thm:threepoints}]

The lower bound follows from \thmref{thm:disc} because $$\chd([P_3,P_3]\times Z_2)=1+2=3.$$

We will work with the space $C_3=C(\C,3)\simeq C(D,3)$ for the remainder of the proof.

To show $\TC(C_3)\le3$ it suffices to find a decomposition of $C_3\times C_3$ into 4 disjoint ENRs such that each of them can be deformed to the diagonal, by \lemref{lem:enr}. 

In the next subsections we will first decompose $C_3\times C_3$ into 4 disjoint ENRs $E_0,E_1,E_2,E_3$ and discuss
some geometric properties of these; then we will describe a motion planner on each $E_i$.

\subsection{Decomposition of $C_3\times C_3$}

First we need a notion of \emph{orientation} for configurations in $C_3$. For this we define a function $\Delta\colon C_3\to \C^*$ by
\[
 \Delta(\set{z_1,z_2,z_3})=(z_1-z_2)^2(z_2-z_3)^2(z_3-z_1)^2,
\]
and let $\delta=\Delta/|\Delta|\colon C_3\to S^1$ be its normalization.

We say that two configurations $x,y\in C_3$ are \textit{cooriented}
if $\delta(x)=\delta(y)$. Let $P\subset C_3\times C_3$ denote the closed subspace of pairs $(x,y)$ for which $x$ and $y$ are cooriented and let $N=C_3\times C_3\setminus P$ denote its complement.

The Lie group $S^1$ on $C_3$ by rotations about the origin. Given a configuration $x\in C_3$ and an element $\theta\in S^1$ the
following formula holds
\begin{align}
\label{deltatheta}
 \delta(\theta\cdot x)=\theta^6\delta(x).
\end{align}

Let $L\subset C_3$ consist of those configurations for which all three points are on a line and let $T=C_3\setminus L$ be its complement. The points in a configuration in $T$ form a nondegenerate triangle; $L\subset C_3$ is closed and $T\subset C_3$ is open.

We define a deformation retraction of $L$ onto the subspace $L_R$ containing configurations of 3 aligned points, one at the origin
and two on the unit circle and opposite to each other. Note that $L_R$ is homeomorphic to a circle and is invariant under rotation. Given a configuration in $L$, we translate it such that the central point ends up at the origin and then slide the two outer points along the line which goes through all three points until they are both at distance 1 from the origin. This defines a deformation retraction
$r_L\colon L\to L_R$. The deformation preserves $\delta$, because the direction determined by any two
points in the configuration remains the same throughout the deformation.

Similarly we define a deformation retraction of $T$ onto the subspace $T_R$ containing configurations of 3 points on the unit
circle that form an equilateral triangle. Note that $T_R$ is also homeomorphic to a circle and invariant under rotation. Given a configuration in $T$, we translate it until the centre of mass coincides with the origin. Then we slide all three points simultaneously along the lines going through the origin until the points land in the unit circle. Finally we rotate the points until they are at equal distance from each other on the unit circle.

More precisely, let $X,Y,Z$ be a configuration of three points on the unit circle, appearing in this order clockwise. Consider
the lengths of the arcs $XY$, $YZ$, $ZX$. If the arcs are all of the same length, then we are done. If there is precisely one arc of minimal length, say $XY$,
then we could slide $X$ and $Y$ at the same speed along the unit circle, gradually increasing the length of $XY$ and decreasing
 both $YZ$ and $ZX$, until the length of $XY$ becomes equal to at least one of the other two arcs.  Therefore, we may assume that there are exactly two
 arcs of minimal length. In this case there is one arc, say $YZ$, which is strictly longer than the other two arcs. Slide both $Y$ and $Z$ at the same speed along the unit circle, gradually decreasing the length of $YZ$ and increasing the lengths of $XY$ and $ZX$, until all three arcs are equal. See figures \ref{angle1} and \ref{angle2}.

\vspace{.3cm}
\begin{minipage}{.48\linewidth}
  \begin{center}
    \includegraphics[scale=.9]{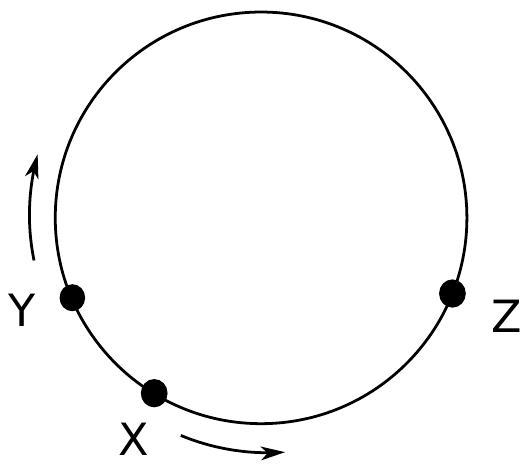}
    \captionof{figure}{First step}
    \label{angle1}
  \end{center}
\end{minipage}%
\begin{minipage}{.48\linewidth}
  \begin{center}
    \includegraphics[scale=.9]{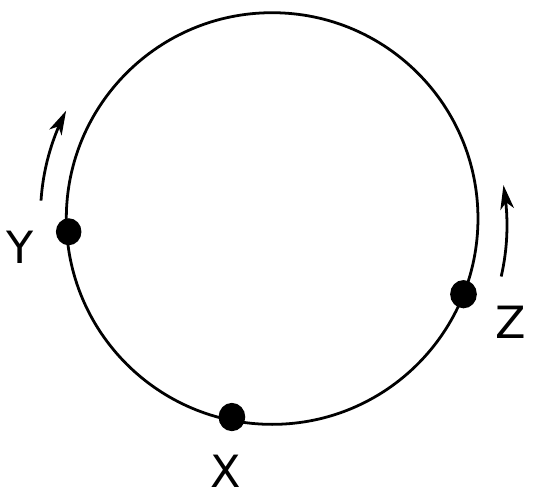}
    \captionof{figure}{Second step}
    \label{angle2}
  \end{center}
\end{minipage}
\vspace{.3cm}

Additionally, we make sure that the above deformation preserves $\delta(x)$ by constantly rotating the configuration $x$ about the origin during the whole process to compensate for the potential change of $\delta(x)$. More precisely, let $H\colon T\times[0,1]\to C_3$ be the homotopy described above, with $H(\cdot,1)\in T_R$, and consider the function $\bar\delta\colon T\times[0,1]\to S^1$
defined by
\[
 \bar\delta(x,t)=\delta(H(x,t))/\delta(x).
\]
Then $\bar\delta(\cdot,0)\colon T\to S^1$ is the constant function $1$ and it admits a lift to the universal covering
$\R\to S^1$, namely the constant function $0$. We can then extend this lift to all positive times, obtaining a map
$\tilde\delta\colon T\times[0,1]\to\R$. Let now $\tilde\rho\colon T\times[0,1]\to \R$ be given by
\[
 \tilde\rho(x,t)=\frac 16 \tilde\delta(x,t)
\]
and denote by $\rho\colon T\times[0,1]\to S^1$ its projection onto $S^1$ along the universal covering map $\R\to S^1$.

Finally consider the homotopy $\bar H\colon T\times [0,1]\to C_3$ given by
\[
 \bar H(x,t)=\left(\rho(x,t)\right)^{-1}\cdot H(x,t).
\]
Then $\bar H$ is a deformation retraction of $T$ onto $T_R$ preserving $\delta$ at all times: this follows easily
from the construction and from formula \ref{deltatheta}.

Denote
$r_T=\bar H(\cdot,1)\colon T\to T_R$.



We are now ready to construct the decomposition into disjoint ENRs as follows.


\begin{itemize}
\item $E_0=P\cap (L\times L)$
\item $E_1=N\cap (L\times L)\sqcup P\cap (T\times L \sqcup L\times T)$
\item $E_2=N\cap (T\times L \sqcup L\times T) \sqcup P\cap (T\times T)$
\item $E_3=N\cap (T\times T)$ 
\end{itemize}

Note that the subspaces $E_i$ are semialgebraic sets and therefore ENRs.

Furthermore, the disjoint unions above are topological, i.e. they form disconnected components inside each $E_i$. This follows from the fact that the disjoint components are relatively open inside each $E_i$. For example $N\cap (L\times L)$ and $P\cap (T\times L)$ are the
intersections of $E_1$ with the open sets $N$ and $T\times C_3$ respectively, and $N\cap (T\times L)$ is the intersection
of $E_2$ with the open set $N\cap (T\times C_3)$.

\subsection{Local motion planners}
We show now that each $E_i$ deformation retracts onto a disjoint union of circles.
First we notice that for $A,B\in\{L,T\}$ the intersection $N\cap (A\times B)$ can be deformed to $P\cap (A\times B)$. Given a pair $(x,y)\in N\cap (A\times B)$, rotate $x$ clockwise about the origin until $x$ and $y$ are cooriented. This can be done continuously
thanks to formula \ref{deltatheta}.

The subspaces $P\cap (L\times L)$ and $P\cap (T\times T)$ deformation
retract to $P\cap (L_R\times L_R)$ and $P\cap (T_R\times T_R)$
respectively, because the retractions $r_L$ and $r_T$ commute with
$\delta$.

The subspaces $P\cap (L_R\times L_R)$ and $P\cap (T_R\times T_R)$ in
turn consist of a disjoint union of three circles and a disjoint union
of two circles respectively, where each circle is an orbit under the
diagonal action of $S^1$ on $C_3\times C_3$. Precisely one orbit in
$P\cap (L_R\times L_R)$ and one orbit in $P\cap (T_R\times T_R)$
already lie in the diagonal of $C_3\times C_3$. The remaining orbits
consist of pairs of lines or pairs of triangles which are at a given
angle from each other ($\pi/3$ or $2\pi/3$ in the case of lines and
$\pi/3$ in the case of triangles to be precise). See figures
\ref{fig:lines} and \ref{fig:triangles}. They can be deformed into the
diagonal by rotating the first configuration in every pair clockwise
about the origin until it is equal to the second configuration in that
pair.

\begin{figure}[htb!]
  \centering
  \includegraphics[scale=0.8]{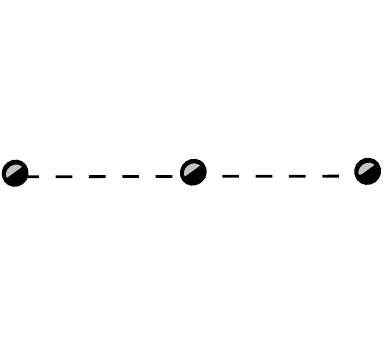}\hspace{1cm}
  \includegraphics[scale=0.8]{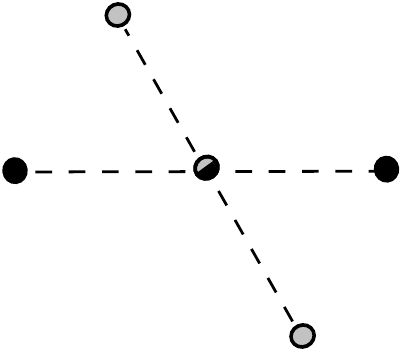}\hspace{1cm}
  \includegraphics[scale=0.8]{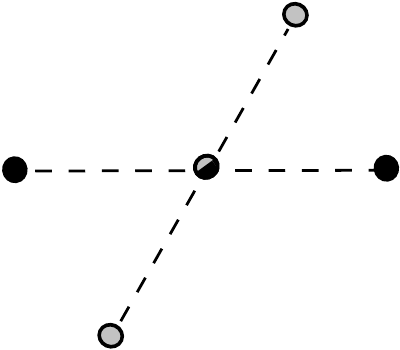}
  \captionof{figure}{Path-components of $P\cap (L_R\times L_R)$ (up to
    rotation).}
  \label{fig:lines}
\end{figure}

\begin{figure}[htb!]
  \centering
  \includegraphics[scale=0.8]{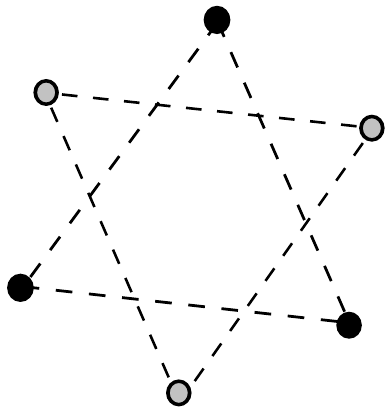}\hspace{1cm}
  \includegraphics[scale=0.8]{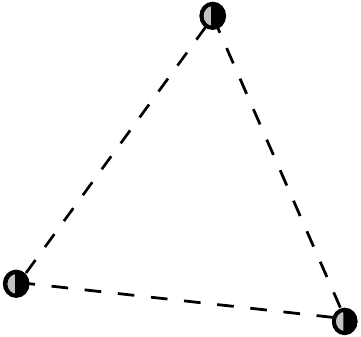}
  \captionof{figure}{Path-components of $P\cap (T_R\times T_R)$ (up to
    rotation).}
  \label{fig:triangles}
\end{figure}

Similarly the space $P\cap (L\times T)$ can be deformed to $P\cap
(L_R\times T_R)$, which consists of one single orbit under the
diagonal $S^1$-action, see Figure \ref{fig:triangle-line}. Specifically it
contains pairs of configurations $(x,y)$, where the points in $y$ form
an equilateral triangle centered at the origin and the points in $x$
lie on a line parallel to one of the sides of said triangle and are
symmetrically distributed around the origin. We move the point in $y$
opposite to the side parallel to $x$ to the origin and the other two
points in $y$ to the corresponding outer points in $x$. The pair
$(x,x)$ is obviously in the diagonal and so we are done.

\begin{figure}[htb!]
  \centering
  \includegraphics{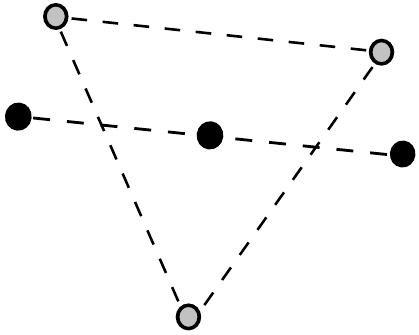}
  \captionof{figure}{The subset $P\cap (T_R\times L_R)$ (up to
    rotation).}
  \label{fig:triangle-line}
\end{figure}

This completes the proof because the deformation can be defined separately on the different disconnected components of each $E_i$.
\end{proof}

\section{Conclusions}\label{sec:conclusions}

The results in this paper can be viewed equivalently as finding the values for the topological complexity of either full braid groups of surfaces or unordered configuration spaces of surfaces, since for aspherical surfaces $\S$
\[
\TC(C(\S,n))=\TC(B_n(\S)).
\]

All the results except the ones which rely on finding explicit motion planners (or equivalently deformations into the diagonal) extend to finite index subgroups of $B_n(\S)$ with the same proofs. To be precise the results which generalize to finite index subgroups are the ones given in the theorems  \ref{thm:disc}, \ref{thm:discsqrt}, \ref{thm:lowergeneral}, \ref{thm:lowerannulus} and \ref{thm:uppergeneral}.

In particular those results apply to the pure braid groups $P_n(\S)$ and the mixed braid groups from \cite{GRM}. Observe that for aspherical surfaces $\S$ the topological complexity $\TC(P_n(\S))$ of the pure braid groups of $\S$ is the same as the topological complexity $\TC(F(\S,n))$ of the ordered configuration spaces of $\S$.

Thus the methods in this paper yield an alternative proof for some of the results given by Cohen and Farber in \cite{CF}, in particular the topological complexity of ordered configuration spaces for all non-closed orientable surfaces (for the ordered configuration spaces of the disc
one can use a slightly modified version of the proof of \ref{thm:upperannulus} to find explicit motion planners). 
Furthermore, it extends their results to all non-closed non-orientable surfaces except the M\"obius band.

It is worth noting that the results in this paper taken together with the results in \cite{CF} are consistent with the possibility that the topological complexities of the ordered and the unordered configuration spaces of a surface coincide, for all surfaces.

The only remaining aspherical surface for which the gap between the lower bound and the upper bound for the topological complexity of its unordered configuration spaces is still arbitrarily large is, perhaps surprisingly, the disc.

If it is in fact true that $\chd([P_n,P_n])=n-2$, then \thmref{thm:disc} would imply $\TC(C(D,n))\ge2n-3$. If additionally the upper bound for $n=3$ given in \thmref{thm:threepoints} generalized to higher $n$, this would completely determine $\TC(C(D,n))$. We make the following

\begin{conj}
If $D$ is the disc, then
\[
\TC(C(D,n))=\TC(B_n)=2n-3.
\]
\end{conj}


\begin{thebibliography}{99}

\bibitem{Art} E. Artin, \emph{Theory of braids}, Ann. of Math. (2) {\bf 48} (1947), 101--126.

\bibitem{BE} R. Bieri, B. Eckmann, \emph{Groups with homological duality generalizing {P}oincar\'e duality}, Invent. Math. {\bf 20} (1973),
   103--124.
   
\bibitem{Coh} D. Cohen, \emph{Topological complexity of classical configuration spaces and related objects}, Topological Complexity and Related Topics, 41--60,
Contemp. Math., 702, Amer. Math. Soc., Providence, RI, (2018).
 
\bibitem{CF} D. Cohen and M. Farber, \emph{Topological complexity of collision-free motion planning on surfaces}, Compos. Math., 147(2):649--660, (2011).

\bibitem{EG} S. Eilenberg, T. Ganea, \emph{On the Lusternik-Schnirelmann category of abstract groups}, Ann. of Math. (2) {\bf 65} (1957), 517--518.

\bibitem{FN} E. Fadell, L. Neuwirth, \emph{Configuration spaces}, Math. Scand. \textbf{10} (1962), 111--118.

\bibitem{Far03} M. Farber, \emph{Topological complexity of motion planning}, Discrete Comput. Geom. {\bf 29} (2003), no. 2, 211--221.

\bibitem{Far06} M. Farber, \emph{Topology of robot motion planning}, Morse theoretic methods in nonlinear analysis and in symplectic topology, NATO Sci. Ser. II Math. Phys. Chem., vol. 217, Springer, Dordrecht, (2006), pp. 185--230.

\bibitem{Far17} M. Farber, \emph{Configuration spaces and robot motion planning algorithms}, Combinatorial and toric homotopy, 263--303, Lect. Notes Ser. Inst. Math. Sci. Natl. Univ. Singap., \textbf{35}, World Sci. Publ., Hackensack, NJ (2018).

\bibitem{FG} M. Farber and M. Grant, \emph{Topological complexity of configuration spaces}, Proc. Amer. Math. Soc. \textbf{137} (2009), no. 5, 1841--1847.

 \bibitem{FGLO} M. Farber, M. Grant, G. Lupton and J. Oprea, \emph{Bredon cohomology and robot motion planning},  preprint (2018) arXiv:1711.10132v2.

\bibitem{FY} M. Farber, S. Yuzvinsky, \emph{Topological Robotics: Subspace Arrangements and Collision Free Motion Planning}, Transl. of AMS, \textbf{212} (2004), 145--156.

\bibitem{Ghr} R. Ghrist, \emph{Configuration spaces and braid groups on graphs in robotics}, in: \emph{Knots, braids, and mapping class groups, papers dedicated to Joan S. Birman (New York, 1998)}, pp. 29--40,
AMS/IP Stud. Adv. Math., vol. 24, Amer. Math. Soc., Providence, (2001).

\bibitem{GLO} M. Grant, G. Lupton and J. Oprea, \emph{New lower bounds for the topological complexity of aspherical spaces}\, Topology Appl. {\bf 189} (2015), 78--91.

\bibitem{GRM} M. Grant and D. Recio-Mitter, \emph{Topological complexity of subgroups of Artin's braid groups}, Topological Complexity and Related Topics, 165--176,
Contemp. Math., 702, Amer. Math. Soc., Providence, RI, (2018).

\bibitem{May} J. P. May, \emph{The Geometry of iterated loop spaces}, Lecture notes in Mathematics 271 (1972), Springer Verlag.

\bibitem{PR} L. Paris, D. Rolfsen, \emph{Geometric subgroups of surface braid groups}, Ann. Inst.
Fourier (Grenoble) (2) {\bf 49} (1999), 417--472.

\bibitem{Sch} S. Scheirer, \emph{Topological complexity of n points on a tree},
To appear in Alg. Geom. Top. arXiv:1607.08185v3 (2018).

\bibitem{Sta} J. R. Stallings, \emph{On torsion-free groups with infinitely many ends}, Ann. of Math. (2) {\bf 88} (1968), 312--334.

\bibitem{Swa} R. Swan, \emph{Groups of cohomological dimension one},
J. Algebra {\bf 12} (1969), 585--610.


\end{thebibliography}
\end{document}